\documentclass[12pt]{amsart}

\usepackage{tgpagella}
\usepackage{euler}
\usepackage[T1]{fontenc}
\usepackage{amsmath, amssymb}
\usepackage[hidelinks]{hyperref}
\usepackage[english]{babel}
\usepackage{mathrsfs}
\usepackage{eucal}
\usepackage[all]{xy}

\newtheorem{thm}{Theorem}[subsection]
\newtheorem*{thm*}{Theorem}
\newtheorem{lem}[thm]{Lemma}
\newtheorem{fact}[thm]{Fact}
\newtheorem{facts}[thm]{Facts}
\newtheorem{prop}[thm]{Proposition}
\newtheorem*{prop*}{Proposition}

\newtheorem{cor}[thm]{Corollary}

\theoremstyle{definition}
\newtheorem{defn}[thm]{Definition}

\newtheorem{remark}[thm]{Remark}
\newtheorem{remarks}[thm]{Remarks}

\newtheorem{example}[thm]{Example}

\def\bb{\mathbb}

\def\bb{\mathbb}

\def\cal{\mathcal}

\newcommand\cU{{\cal U}}

\newcommand\bbN{{\mathbf N}}

\def \aA{{}_\alpha A}
\def \tT{\tilde{T}}
\def \st{\operatorname{st}}

\begin{document}


\title{A nonstandard take on central sets}
\author{Isaac Goldbring}
\thanks{Goldbring's work was partially supported by NSF CAREER grant DMS-1349399.}

\address {Department of Mathematics, University of California, Irvine, 340 Rowland Hall (Bldg.\# 400), Irvine, CA, 92697-3875.}
\email{isaac@math.uci.edu}
\urladdr{http://www.math.uci.edu/~isaac}

\begin{abstract}
We present the basic theory of central subsets of semigroups from the nonstandard perspective.  A key feature of this perspective is the replacement of the algebra of ultrafilters with the algebra of elements of iterated hyperextensions, a technique first employed by Mauro Di Nasso.
\end{abstract}

\maketitle

\tableofcontents

\section{Introduction}

One of the main themes of the subject known as Ramsey theory on the natural numbers is the study of \textbf{partition regular} families of subsets of $\bb N$, where $\cal F\subseteq \cal P(\bb N)$ is partition regular if:  whenever $A\in \cal F$ and $A=\bigsqcup_{i=1}^n A_i$ is a partition of $A$ into finitely many pieces, then there is $i\in \{1,\ldots,n\}$ such that $A_i\in \cal F$.\footnote{Sometimes this is phrased in terms of \textbf{colorings}:  if a member of $\cal F$ is colored with finitely many colors, then there is a monochromatic subset belonging to $\cal F$.  For this to really be an equivalence, $\cal F$ needs to be closed under supersets, which it often is.}  If, in the preceding definition, we only look at finite partitions of $\bb N$ itself, then we say that the family is merely \textbf{weakly partition regular.}

Although there are many important partition regular families, two such families will play an important role in this paper:

\begin{defn}

\

\begin{enumerate}
\item $A\subseteq \bb N$ is called \textbf{piecewise syndetic} if there is a finite $G\subseteq \bb N$ such that $A+G$ contains arbitrarily long intervals.
\item $A\subseteq \bb N$ is an \textbf{FS-set} if there is an infinite set $X$ such that FS$(X)\subseteq A$, where FS$(X):=\{\sum_{x\in F} x \ : F\subseteq X\text{ is nonempty and finite}\}$.
\end{enumerate}
\end{defn}

The fact that the family of piecewise syndetic subsets of $\bb N$ is partition regular is known as Brown's lemma, although its proof is quite straightforward (and especially elegant from the nonstandard perspective \cite[Corollary 11.19]{book}).  On the other hand, the partition regularity of the family of FS-sets is a deeper result known as \textit{Hindman's theorem} (although technially the original version of Hindman's theorem only established weak partition regularity) and is a cornerstone result in the area.

It is only natural to seek a partition regular family of subsets of $\bb N$ contained in the intersection of the aforementioned two families.  We should note that the intersection of the aforementioned two families is not itself partition regular:

\begin{example}
Let $A$ be a piecewise syndetic set that is not an FS-set (e.g. the set of odd numbers) and let $B$ be an FS-set that is not piecewise syndetic (e.g. FS$(X)$ for $X\subseteq\bb N$ sufficiently sparse).  It remains to note that $A\cup B$ is both piecewise syndetic and an FS-set.  
\end{example}

This note is about the class of \textbf{central} subsets of $\bb N$, which is indeed a partition regular family of subsets of $\bb N$ and each central subset of $\bb N$ is both piecewise syndetic and an FS-set.  As we will point out later, central sets contain a lot more structure than merely being both piecewise syndetic and an FS-set.

The key to defining central sets is to give ultrafilter characterizations of piecewise syndetic sets and FS-sets:  $A\subseteq \bb N$ is piecewise syndetic (resp. an FS-set) if and only if it belongs to a \textbf{minimal} (resp. \textbf{idempotent}) ultrafilter; these terms will be defined shortly.  The family of central sets can thus be defined to be those sets that belong to an ultrafilter that is both minimal and idempotent.  The partition regularity of the family of central sets is now immediate from this characterization.

The account given above is actually revisionist history.  Indeed, in \cite{Furst}, Furstenberg introduced the family of central subsets of $\bb N$ in connection with his work in dynamical systems.  His definition is, at first glance, completely different from the one given above and will be discussed in the last section of this note.  Furstenberg showed that this class is weakly partition regular\footnote{He also mentions, without proof, that any finite coloring of a central set contains a monochromatic central subset.  At the time he was unaware of the fact that central sets were closed under supersets, a fact first pointed out by Hindman using the ultrafilter characterization.} and that every central set contains arbitrarily long arithmetic progressions.\footnote{He mentions, without proof, that central sets can be shown to piecewise syndetic, which, by van der Waerden's theorem, would also yield that central sets contain arbitrarily long arithmetic progressions.}  He then proved a theorem that shows that central sets contain a lot of extra structure (which implies, in particular, that they are FS-sets); this theorem is now a special case of a much more general theorem called the \emph{Central Sets Theorem}, which will be discussed in Section 5.  It was only later on that Bergelson and Hindman realized that the conclusion of the Central Sets Theorem should also hold for members of minimal idempotent ultrafilters.

At the meeting ``Combinatorics meets ergodic theory'' at BIRS in 2015, Randall McCutcheon asked me if there is a nonstandard perspective on the theory of central sets.  It is the purpose of this note to give such a perspective.  Given the fact that every ultrafilter can be represented as a ``hyper-principal'' ultrafilter with a nonstandard generator, the existence of such a perspective should not be so surprising.  Moreover, using the replacement of ``algebra in the space of ultrafilters on $\bb N$'' with ``algebra in the space of \emph{iterated nonstandard extensions} of $\bb N$'' à la Mauro Di Nasso \cite{DN}, many of the arguments given in \cite{HS} laying the foundation for the basic theory of central sets become much shorter and more natural in the nonstandard context.

We now provide an outline of the contents of this article.  In Section 2, we gather the necessary preliminaries from nonstandard analysis, focusing mainly on the nonstandard perspective on ultrafilters and the use of iterated hyperextensions.  We also give the nonstandard proof of Hindman's theorem as it is an easier version of many arguments that appear later in this note.  In Section 3, we give the nonstandard account of minimal ideals and use this to prove some of the basic facts about central sets.  In Section 4, we give the combinatorial (that is, ultrafilter-free) description of central sets.  One of the ingredients of this description, namely the notion of a \emph{collectionwise piecewise syndetic family} of subsets of $\bb N$, becomes especially transparent from the nonstandard perspective.  In Section 5, we state and prove the aforementioned Central Sets Theorem and indicate some of its consequences.  In the final section, we present Furstenberg's original dynamical definition of central sets and establish the equivalence with the ultrafilter formulation.

We reiterate that most, if not all, of the arguments contained in this note are the nonstandard versions of the arguments appearing in the fantastic book \cite{HS}, which contains a lot more information about central sets than we present here.  We do believe, however, that the nonstandard versions of the arguments are aesthetically cleaner and computationally more natural.  Two other references that proved useful during the writing of this note are Hindman's suvey on central sets \cite{Hindman} and Bergelson's survey on ultrafilters in combinatorial number theory \cite{Bergelson}.

We end this introduction with some conventions maintained throughout this note.
\begin{itemize}
\item $(S,\cdot)$ denotes an arbitrary semigroup.
\item We set $\bb N:=\{1,2,3,\ldots\}$ to be the set of natural numbers which, in this context, is assumed \textbf{not} to contain $0$.
\item For a set $X$, we let $\cal P(X)$ denote the power set of $X$ and $\cal P_f(X)$ denote the set of finite subsets of $X$.
\item When we write $A=\bigsqcup_{i=1}^n A_i$, this indicates that the set $A$ has been partitioned into the disjoint subsets $A_1,\ldots,A_n$.
\item For $m\in \bb N$, we let $\bb N^{[m]}$ denote the $m$-element subsets of $\bb N$, which we often identify with increasing sequences $t(1)<t(2)<\cdots<t(m)$.
\end{itemize}

\section{Preliminaries}

For the sake of brevity, we assume that the reader is familiar with the basics of nonstandard analysis.  A recent monograph \cite{book}, written with applications to Ramsey theory and combinatorial number theory in mind, also contains a complete introduction to the subject.  We only mention here some crucial facts needed for the remainder of this article.

As usual, we assume that our nonstandard extension is as saturated as necessary to make the arguments below valid.

\subsection{Nonstandard generators of ultrafilters and iterated hyperextensions}

We let $\beta S$ denote the Stone-\v{C}ech compactification of the discrete space $S$.  It can be identified with the space of ultrafilters on $S$, where a basis of clopen sets for the topology is given by $\bar A:=\{\cU \in \beta S \ : \ A\in \cU\}$ for $A\subseteq S$.  We view $S$ as a subset of $\beta S$ by identifying $s\in S$ with the principal ultrafilter $\cU_s$ generated by $s$.

The semigroup operation on $S$ extends to a semigroup operation on $\beta S$ determined by declaring, for $\cU,\cal V\in \beta S$ and $A\subseteq S$, that
$$A\in \cU \cdot \cal V \Leftrightarrow \{s\in S \ :  s^{-1}A\in \cal V\}\in \cU.$$  Here, $s^{-1}A:=\{t\in S \ : \ st\in A\}$.  Although the extended semigroup operation on $\beta S$ need not be continuous, we do have that the maps 
$$\cU\mapsto \cU_s\cdot \cU,\cU\mapsto \cU\cdot \cal V:\beta S\to \beta S$$ are continuous for each $s\in S$ and $\cal V\in \beta S$.

Given $\alpha\in S^*$ (the nonstandard extension of $S$), set $\cU_\alpha:=\{A\subseteq S \ : \ \alpha\in A^*\}$.  It is easy to see that $\cU_\alpha$ is an ultrafilter on $S$ and that this notation agrees with the notation above when $\alpha$ is a standard element of $S$.  Conversely, given $\cU\in \beta S$, there is (assuming sufficient saturation) some $\alpha\in \bigcap_{A\in \cU}A^*$; for such an $\alpha$, we have $\cU=\cU_\alpha$.

We let $\pi:S^*\to \beta S$ be the canonical surjection given by $\pi(\alpha):=\cU_\alpha$.  While we just obsered that $\pi$ is surjective, it is not (in general) injective, that is, there may be many nonprincipal generators for a given ultrafilter.  We define an equivalence relation $\sim$ on $S^*$ by setting $\alpha\sim \beta$ if $\cU_\alpha=\cU_\beta$; in other words, $\alpha\sim \beta$ if and only if:  for every $A\subseteq \cU$, we have $\alpha\in A^*\Leftrightarrow \beta\in A^*$.  It follows that $\pi$ descends to a bijection $\bar \pi:S^*$/$\sim \ \to \beta S$.

The \textbf{u-topology} on $S^*$ has as a basis of clopen sets the sets $A^*$ for $A\subseteq S$.  Note that the u-topology on $S^*$ is compact but not Hausdorff and, in fact, the map $\bar \pi$ witnesses that $\beta S$ is homeomorphic to the Hausdorff separation of $S^*$.  Although $S^*$ carries other natural topologies, in this note, all references to topological notions in $S^*$ will be with respect to the u-topology. 

Clearly the nonstandard extension of the semigroup operation on $S$ is a semigroup operation on $S^*$.  The na\"ive expectation would be that $\pi$ is a semigroup homomorphism, that is, $\cU_{\alpha\cdot \beta}=\cU_\alpha\cdot\cU_\beta$.  This is unfortunately not the case (see \cite[Example 3.8]{book} for concrete counter-examples).  However, there is still a viable formula along these lines whose validity allows the nonstandard method to be applicable to the algebra of ultrafilters.

Fix $\alpha,\beta\in S^*$ and $A\subseteq S$.  We set $$A\cdot \cU_\beta^{-1}:=\{s\in S \ : \ s^{-1}A\in \cU_\beta\}=\{s\in S \ : \ s\cdot \beta\in A^*\}.$$  By the definition of the semigroup operation on $\beta S$, we have that $$A\in \cU_\alpha\cdot \cU_\beta \Leftrightarrow A\cdot \cU_\beta^{-1}\in \cU_\alpha \Leftrightarrow \alpha\in (A\cdot \cU_\beta^{-1})^*.$$  Working na\"ively (and motivated by some kind of transfer principle), the latter equivalent should in turn be equivalent to $\alpha \cdot \beta^*\in A^{**}$.  Of course, for this to make any sense, one needs to give meaning to the objects $\beta^*$ and $A^{**}$. 

One can indeed give concrete meaning to objects like $\beta^*$ and $A^{**}$.  This idea was first pursued by Mauro Di Nasso in \cite{DN}, where he used this technique to give an ultrafilter generaliztion of Rado's classical theorem on parition regularity of linear equations.  One works in a framework for nonstandard analysis where one can \emph{iterate} the $*$ operation, whence $\beta^*$ above is an element of $S^{**}$ and $A^{**}$ is a subset of $S^{**}$.  There is an obvious transfer principle between one level of the tower of iterated nonstandard extensions and the next level.  For complete details, see \cite{DN} or \cite{book}, the latter of which contains many applications of this technique to Ramsey theory.  Admittedly this approach takes some getting used to (e.g. unlike the usual convention that $s^*=s$ for $s\in S$, we now have that $\alpha^*\not=\alpha$ for $\alpha\in S^*\setminus S$); however, once one is familiarized with this framework\footnote{This should hopefully be the case by the time you have finished reading this note.}, it proves to be extremely convenient.  

We set up some notation concerning iterated nonstandard extensions:  for $k\in \bb N$, we let $S^{(k)*}$ denote the $k^{\text{th}}$ iterate of the *-map applied to $S$ and we let $S^{(\infty)*}:=\bigcup_{n\in \omega} S^{(n)*}$.  (Here, $S^{(0)*}:=S$.)  Many of the ideas from earlier in this section can be adapted to this extended framework.  For example, given $\alpha\in S^{(*)k}$, we set $\cU_\alpha:=\{A\subseteq S \ : \ \alpha\in A^{(k)*}\}$, which is again an ultrafilter on $S$, and for $\alpha,\beta\in S^{(\infty)*}$, we write $\alpha\sim \beta$ if and only if $\cU_\alpha=\cU_\beta$.

Returning to the earlier context:  for $\alpha,\beta\in S^*$ and $A\subseteq S$, we now have
$$A\in \cU_\alpha\cdot \cU_\beta\Leftrightarrow \alpha\cdot \beta^*\in A^{**}\Leftrightarrow A\in \cU_{\alpha\cdot \beta^*}.$$  In other words, $\cU_{\alpha}\cdot \cU_{\beta}=\cU_{\alpha\cdot \beta^*}$.

\subsection{Idempotent elements and FP-sets}

Equipped with the framework of iterated nonstandard extensions, we can now give an extremely clean proof of Hindman's theorem.  We first need the following fundamental fact about $\beta S$, which follows from a straightforward application of a classical theorem of Ellis (see \cite[Thm 2.5]{HS}):

\begin{fact}
For every nonempty compact subsemigroup $M$ of $\beta S$, there is $\cU\in M$ such that $\cU\cdot \cU=\cU$.
\end{fact}

An ultrafilter as in the statement of the previous fact is called \textbf{idempotent}.  Clearly a nonstandard generator of an idempotent ultrafilter is an idempotent element of $S^*$ in the sense of the following

\begin{defn}
$\alpha\in S^*$ is \textbf{idempotent}\footnote{In other works, such elements are called u-idempotent. We prefer the current terminology even though it is potentially confusing as $S^*$ is itself a semigroup and thus there is already the usual algebraic notion of idempotent elements of $S^*$.  To avoid confusion, we will never speak of algebraic idempotent elements of $S^*$.  Note that, by transfer, if $S^*$ has an algebraic idempotent element, then so does $S$.} if $\alpha\cdot \alpha^*\sim \alpha$.
\end{defn}

The nonstandard version of a subsemigroup of $\beta S$ is the following:

\begin{defn}
$T\subseteq S^*$ is a \textbf{u-subsemigroup} if, for every $\alpha,\beta\in T$, there is $\gamma\in T$ such that $\gamma\sim \alpha\cdot \beta^*$.
\end{defn}

The following is clear:

\begin{lem}
$T\subseteq S^*$ is a u-subsemigroup if and only if $\pi(T)$ is a subsemigroup of $\beta S$.
\end{lem}

\begin{cor}
Every nonempty closed u-subsemigroup of $S^*$ contains an idempotent element.
\end{cor}

We next aim to prove Hindman's theorem for an abitrary semigroup.  We should first define the arbitrary semigroup analog of an FS-set:

\begin{defn}
For a sequence $\langle s_n \rangle_{n=1}^\infty$ from $S$, we set
$$FP(\langle s_n\rangle_{n=1}^\infty):=\left\{\prod_{i=1}^k s_{j_i} \ : \ j_1<j_2<\cdots<j_k, k>0\right\}.$$  (One defines the notion FP$(\langle s_n\rangle_{n=1}^m)$ in an analogous fashion.)  We say that $A\subseteq S$ is an \textbf{FP-set} if there is a sequence $\langle s_n\rangle_{n=1}^\infty$ from $S$ such that FP$(\langle s_n\rangle_{n=1}^\infty)\subseteq A$.
\end{defn}

We now wish to show that if $\alpha$ is idempotent and $\alpha\in A^*$, then $A$ is an FP-set.  The following definitions will become useful:

\begin{defn}
For $A\subseteq S$ and $\alpha\in S^*$, we set 
\begin{itemize}
\item $A_\alpha:=\{s\in S \ : \ s\cdot \alpha\in A^*\}$ and
\item $\aA:=A\cap A_\alpha$.
\end{itemize}
\end{defn}

The following lemma is immediate from the definitions:

\begin{lem}
$\alpha\in S^*$ is idempotent if and only if:  for every $A\subseteq S$, if $\alpha\in A^*$, then $\alpha\in \aA^*$.  In this case, if $s\in A_\alpha$ (resp. $s\in \aA$), then $s\alpha\in A_\alpha^*$ (resp. $s\alpha\in \aA^*$).
\end{lem}

We can now prove:

\begin{prop}\label{idemFP}
Suppose that $\alpha\in S^*$ is idempotent and $\alpha\in A^*$.  Then $A$ is an FP-set.
\end{prop}

\begin{proof}
We recursively construct a sequence $\langle x_n \rangle_{n=1}^\infty$ such that, for all $m\in \bb N$, we have FP$(\langle x_n \rangle_{n=1}^m)\subseteq \aA$.  Since $\alpha\in \aA^*$, there is $x_1\in \aA$.  Suppose now that $\langle x_n \rangle_{n=1}^m$ has been defined with FP$(\langle x_n \rangle_{n=1}^m)\subseteq \aA$.  By the previous lemma, we have FP$(\langle x_n \rangle_{n=1}^m)\cdot \alpha\subseteq \aA^*$.  By transfer, there is $x_{m+1}\in \aA$ with FP$(\langle x_n \rangle_{n=1}^m)\cdot x_{m+1}\subseteq \aA$, whence $x_{m+1}$ is as desired.
\end{proof}

Note that, if $\alpha\notin S$ (e.g. when $S$ has no idempotent elements), then we can assume that the sequence above is injective.  In particular, when $S=\bb N$, we can take the sequence above to be increasing.

\begin{cor}[Hindman's theorem]
Suppose that $\bb N=\bigsqcup_{i=1}^n A_i$.  Then some $A_i$ is an FP-set.
\end{cor}

\begin{proof}
Fix an idempotent $\alpha \in S^*$ and take $A_i$ with $\alpha\in A_i^*$. 
\end{proof}

There is a converse to Proposition \ref{idemFP}.  For a nonstandard proof, see, for example, \cite[Lemma 9.5]{book}.

\begin{prop}
Suppose that $A$ is an FP-set.  Then there is an idempotent $\alpha\in A^*$.
\end{prop}

\begin{cor}[Strong Hindman's Theorem]
The notion of being an FP-set is partition regular:  if $A$ is an FP-set and $A=\bigsqcup_{i=1}^n A_i$, then some $A_i$ is an FP-set. 
\end{cor}

\subsection{Three notions of largeness}

The following notions of largeness will appear throughout this note:

\begin{defn}
Suppose that $A\subseteq S$.  We say that:
\begin{enumerate}
\item $A$ is \textbf{thick} if, for every finite $F\subseteq S$, there is $t\in S$ such that $Ft\subseteq A$.  
\item $A$ is \textbf{syndetic} if there is a finite $G\subseteq S$ such that $S=G^{-1}A$. 
\item $A$ is \textbf{piecewise syndetic} if there is a finite $G\subseteq S$ such that, for every finite $L\subseteq S$, there is $a\in S$ with $La\subseteq G^{-1}A$.  
\end{enumerate}
\end{defn}

In the above definitions, $G^{-1}A:=\{s\in S \ : \ gs\in A \text{ for some }g\in G\}$.

Here are the nonstandard equivalents:

\begin{lem}
Suppose that $A\subseteq S$.
\begin{enumerate}
\item $A$ is thick if and only if there is $\alpha\in S^*$ such that $S\alpha\subseteq A^*$.  
\item $A$ is syndetic if and only if $S^*\subseteq S^{-1}A^*$ if and only if there is a finite $G\subseteq S$ such that $S^*\subseteq G^{-1}A^*$.
\item $A$ is piecewise syndetic if and only if there is $\alpha\in S^*$ and finite $G\subseteq S$ such that  $S\alpha\subseteq G^{-1}A^*$.
\end{enumerate}
\end{lem}

\section{Minimal elements and central sets}

\subsection{Facts about minimal ideals}

Recall that a subset $I$ of $\beta S$ is a \textbf{left} (resp. \textbf{right}) \textbf{ideal} if for all $\cU\in \beta S$ and $\cal V\in I$, we have $\cU \cdot \cal V\in I$ (resp. $\cal V\cdot \cU\in I$).  $I$ is an \textbf{ideal} if it is both a left and right ideal.  A left (resp. right) ideal $I$ is \textbf{minimal}  if there is no left (resp. right) ideal properly contained in $I$.

We will need the following facts about minimal ideals in $\beta S$.  None of these facts are especially difficult and can be found in \cite{HS}.
\begin{facts}

\

\begin{enumerate}
\item Every left ideal in $\beta S$ contains a minimal left ideal.
\item Minimal left ideals are closed.
\item $\beta S$ has a smallest ideal $K(\beta S)$, that is, $K(\beta S)$ is contained in all ideals of $\beta S$.
\item $K(\beta S)$ is the union of the minimal left ideals of $\beta S$ and is also the union of the minimal right ideals of $\beta S$.
\end{enumerate}
\end{facts}

We now study the corresponding nonstandard perspective:

\begin{defn}
$\alpha\in S^{(\infty)*}$ is \textbf{minimal} if $\cU_\alpha\in K(\beta S)$.
\end{defn}

The following is obvious from the fact that $K(\beta S)$ is an ideal:

\begin{lem}
Suppose that $\alpha\in S^*$ is minimal and $\beta,\gamma\in S^*$ are arbitrary.  Then the following are minimal:  $\beta \alpha^*$, $\alpha\gamma^*$, $\beta\alpha^*\gamma^{**}$.
\end{lem}

\begin{defn}
$J\subseteq S^*$ is a \textbf{left $u$-ideal} if, for every $\alpha\in S^*$ and $\beta \in J$, there is $\gamma\in J$ such that $\alpha\cdot \beta^*\sim \gamma$.
\end{defn}

Call $J\subseteq S^*$ \textbf{full} if it is closed under $\sim$.  Recall that $\pi:S^*\to \beta S$ is the canonical projection map $\pi(\alpha)=\cU_\alpha$  The following is obvious:

\begin{lem}
$J\subseteq S^*$ is a left $u$-ideal if and only if $\pi(J)$ is a left ideal of $\beta S$.  In particular, if $I$ is a left ideal of $\beta S$, then $\pi^{-1}(I)$ is a full left u-ideal of $S^*$.
\end{lem}

\begin{cor}
If $J\subseteq S^*$ is a left $u$-ideal, then $J$ contains a minimal element of $S^*$.
\end{cor}

The following is also obvious:

\begin{fact}
If $I\subseteq \beta S$ is a minimal left ideal, then $I=\beta S \cdot \cU$ for every $\cU\in I$.
\end{fact}

The previous fact says that every minimal left ideal of $\beta S$ is principal and every element of the ideal is a generator.  We need the nonstandard equivalent.  For $\alpha\in S^*$, set $$\cal L(S^*\cdot \alpha^*):=\{\beta\in S^* \ : \ \beta\sim \gamma\cdot \alpha^* \text{ for some }\gamma\in S^*\}.$$  

\begin{lem}
$\cal L(S^*\cdot \alpha^*)=\pi^{-1}(\beta S\cdot \cU_\alpha)$.  Consequently, $\cal L(S^*\cdot \alpha^*)$ is a full left $u$-ideal.
\end{lem}


\begin{defn}
We say that $J\subseteq S^*$ is a \textbf{minimal left $u$-ideal} if and only if $\pi(J)$ is a minimal left ideal of $\beta S$.
\end{defn}

\begin{lem}
If $J\subseteq S^*$ is a minimal left $u$-ideal, then $J\subseteq \cal L(S^*\cdot \alpha^*)$ for every $\alpha \in J$.
\end{lem}

\begin{proof}
The condition ``$J\subseteq \cal L(S^*\cdot \alpha^*)$ for every $\alpha\in J$'' is equivalent to the condition ``$\pi(J)=\beta S\cdot \cU_\alpha$ for every $\alpha\in J$.''
%
\end{proof}

Finally, we record the following consequence of the fact that every nonempty minimal left ideal of $\beta S$ is in particular a compact subsemigroup of $\beta S$:

\begin{lem}
Every nonempty minimal left u-ideal contains an idempotent element.
\end{lem}

The analogous definitions and results for right $u$-ideals should be apparent to the reader.

Given all of the above preparation, we can now give the following useful characterization of minimal elements of $S^*$:

\begin{thm}\label{pwsminimal}
Given $\alpha\in S^*$, the following are equivalent:
\begin{enumerate}
\item $\alpha$ is minimal.
\item For all $A\in \cU_{\alpha}$, $A_\alpha$ is syndetic.
\item For all $\beta\in S^*$, there is $\gamma\in S^*$ such that $\alpha\sim \gamma\beta^*\alpha^{**}$.
\end{enumerate}
\end{thm}

\begin{proof}
(1) implies (2):  Suppose that $\alpha$ is minimal and let $L$ be a minimal left $u$-ideal containing $\alpha$.  Fix $\beta\in L$.  Then $L\subseteq \cal L(S^*\cdot \beta^*)$ whence there is $\gamma\in S^*$ such that $\alpha\sim \gamma\cdot \beta^*$.  Thus, if $A\in \cU_\alpha$, then $\alpha\in A^*$ whence $\gamma\cdot \beta^*\in A^{**}$, and so, by transfer, there is $t\in S$ such that $t\cdot \beta\in A^*$.  Since $\beta\in L$ is arbitrary, it follows that $L\subseteq S^{-1}A^*$.  Thus, for any $\delta\in S^*$, we have that $\delta\alpha^*\in t^{-1}A^{**}$ for some $t\in S$, whence $t\delta\in A_\alpha^*$.  It follows that $S^*\subseteq S^{-1}A_\alpha^*$, so $A_\alpha$ is syndetic.    

(2) implies (3):  Fix $\beta\in S^*$ and $A\in \cU_\alpha$.  (2) implies that there is $t\in S$ such that $t\beta\in A_\alpha^*$, that is, $\alpha^*\in (t\beta)^{-1}A^{**}$.  By saturation, there is $\gamma\in S^*$ such that, for all $A\in \cU_\alpha$, $\alpha^{**}\in (\gamma\beta^*)^{-1}A^{***}$, that is, $\gamma\beta^*\alpha^{**}\in A^{***}$, whence $\cU_{\alpha}\subseteq \cU_{\gamma\beta^*\alpha^{**}}$ and hence $\cU_{\alpha}= \cU_{\gamma\beta^*\alpha^{**}}$, as desired.


(3) implies (1):  Fix minimal $\beta\in S^*$ and take $\gamma\in S^*$ such that $\alpha\sim \gamma\beta^*\alpha^{**}$.  Since the latter element is minimal, so is $\alpha$.
\end{proof}

\begin{remarks}

\

\begin{enumerate}
\item The proof of (1) implies (2) in the standard context uses the compactness of minimal left ideals.  The nonstandard approach seems to avoid this.
\item Item (2) above is similar to the property of idempotent ultrafilters in that sets in the ultrafilter have large ultrafilter shits, where large in the former means in the ultrafilter and in the latter means syndetic.
\end{enumerate}
\end{remarks}

\subsection{Piecewise syndetic sets and central sets}

The following result provides the crucial link between piecewise syndetic sets and minimal elements:

\begin{thm}
Suppose that $A\subseteq S$.  Then $A$ is piecewise syndetic if and only if there is a minimal $\alpha\in A^*$.
\end{thm}

\begin{proof}
First suppose that $A$ is piecewise syndetic.  Take finite $G\subseteq S$ and $\alpha\in A^*$ such that $S\cdot \alpha\subseteq G^{-1}A^*$.  By transfer, $S^*\cdot \alpha^*\subseteq G^{-1}A^{**}$.  Take minimal $\beta\in S^*$ such that $\beta\in \cal L(S^*\cdot \alpha^*)$, so $\beta\in G^{-1}A^{*}$.  Take $t\in G$ such that $t\beta\in A^*$; it remains to notice that $t\beta$ is minimal. 

Now suppose that $\alpha\in S^*$ is minimal and $\alpha\in A^*$.  By the previous theorem, $A_\alpha$ is syndetic, so there is finite $G\subseteq S$ such that $S=G^{-1}A_\alpha$.  It follows that $S\alpha\subseteq G^{-1}A^*$, whence $A$ is piecewise syndetic.   
\end{proof}

\begin{cor}
$\alpha\in S^*$ is in the closure of the minimal elements if and only if every $A\in \cU_\alpha$ is piecewise syndetic.
\end{cor}

We now come to the central definition of this note:

\begin{defn}

\

\begin{enumerate}
\item $\alpha\in S^*$ is a \textbf{minimal idempotent} if it is both minimal and idempotent.
\item $A\subseteq S$ is \textbf{central} if there is a minimal idempotent $\alpha\in A^*$.
\end{enumerate}
\end{defn}

The following is immediate from the definition:

\begin{prop}
The notion of being central is partition regular.
\end{prop}

By our earlier discussions, every central set is both piecewise syndetic and an FP-set.  By partition regularity of being central, the example from the introduction is a piecewise syndetic FP-set that is not central.

Although piecewise syndetic sets need not be central, we now show that every piecewise syndetic set has a shift that is central:

\begin{thm}
For $A\subseteq S$, the following are equivalent:
\begin{enumerate}
\item $A$ is piecewise syndetic.
\item $\{x\in S \ : \ x^{-1}A \text{ is central}\}$ is syndetic.
\item $\{x\in S \ : \ x^{-1}A \text{ is central}\}$ is nonempty.
\end{enumerate}
\end{thm}

\begin{proof}
(1) implies (2):  Let $B:=\{x\in S \ : \ x^{-1}A \text{ is central}\}$.  Fix $\delta\in S^*$.  We need to find $t\in S$ such that $t\delta\in B^*$.  Take minimal $\alpha\in A^*$.  Let $L$ be a minimal left u-ideal with $\alpha\in L$.  Let $\beta\in L$ be idempotent.  Since $A_\beta\subseteq B$, it suffices to find $t\in S$ such that $t\delta\in A_\beta^*$.  Since $L\subseteq \cal L(S^*\cdot \beta^*)$,  there is $\gamma\in S^*$ such that $\alpha\sim \gamma\cdot \beta^*$, whence 
$$\alpha\cdot \beta^*\sim \gamma \cdot \beta^*\cdot \beta^{**}\sim \gamma\cdot \beta^*\sim \alpha.$$  It follows that $\alpha\cdot \beta^*\in A^{**}$, whence there is $s\in S$ such that $\beta\in s^{-1}A^*$.  Since $\beta$ is minimal, by Theorem \ref{pwsminimal}, we have that $(s^{-1}A)_\beta$ is syndetic, whence there is $u\in S$ such that $u\delta\in (s^{-1}A)_\beta^*$, that is $\beta^* \in (u\delta)^{-1}s^{-1}A^{**}=(su\delta)^{-1}A^{**}$, whence $su\delta\in A_\beta^*$.  Thus, setting $t:=su$, we have $t\delta\in A_\beta^*$, as desired.

(2) implies (3) is trivial.

(3) implies (1):  Take $x\in S$ such that $x^{-1}A$ is central and let $\alpha\in x^{-1}A^*$ be a minimal idempotent.  Then $x\alpha\in A^*$ is minimal, so $A$ is piecewise syndetic.
\end{proof}

We finish this subsection by showing that thick sets are central.  First:

\begin{thm}
$A\subseteq S$ is thick if and only if there is a left $u$-ideal $L\subseteq A^*$.
\end{thm}

\begin{proof}
First suppose that $A$ is thick, so $S\alpha\subseteq A^*$ for some $\alpha\in S^*$.  It follows that $S^*\alpha^*\subseteq A^{**}$, whence $\cal L(S^*\alpha^*)$ is a left $u$-ideal contained in $A^*$.

Conversely, suppose that $L\subseteq A^*$ is a left u-ideal.  Fix $\alpha\in L$.  Then $S\alpha\subseteq L\subseteq A^*$, as desired.
\end{proof}

Since every left $u$-ideal contains a minimal left $u$-ideal, we have:
\begin{cor}
Thick sets are central.
\end{cor}

\subsection{Addition and multiplication}

In this subsection, we work with the semigroups $(\bb N,+)$ and $(\bb N,\cdot)$.  We then use the adjectives ``additive'' and ``multiplicative'' to make it clear which semigroup we are speaking about.

We consider the following two sets (using the same notation as found in \cite{HS}):

\begin{defn}

\

\begin{enumerate}
\item $\Gamma$ is the closure of the set of additively idempotent elements of $\bb N^*$.
\item $\bb M$ is the closure of the set of additively minimal idempotent elements of $\bb N^*$.
\end{enumerate}
\end{defn}

Thus, $\alpha \in \Gamma$ (resp. $\alpha\in \bb M$) if and only if, whenever $\alpha\in A^*$, then $A$ is an FS-set (resp. $A$ is an additively central set).

We use the notation $A/n:=\{m\in \bb N \ : \ mn\in A\}$.

\begin{prop}
$\Gamma$ and $\bb M$ are multiplicative left u-ideals of $\bb N^*$.
\end{prop}

\begin{proof}
Take $\alpha\in \bb N^*$, $\beta \in \Gamma$, and suppose that $\alpha\cdot \beta^*\in A^{**}$; we need to show that $A$ is an FS-set.  Take $n\in \bb N$ such that $n\cdot \beta \in A^*$, so $\beta\in A^*/n$.  It follows that $A/n$ is an FS-set, whence so is $A$.  Indeed, take an additively idempotent $\gamma\in A/n$; it suffices to see that $n\gamma$ is also additively idempotent, which follows from the calculation
$$n\gamma+(n\gamma)^*=n\gamma+n\gamma^*=n(\gamma+\gamma^*)\sim n\gamma.$$

If $\beta \in \bb M$, then $A/n$ in the above paragraph is additively central, whence $\gamma$ can be chosen to be a minimal idempotent element.  It just remains to observe that $n\gamma$ is also minimal, whence $A$ is additively central.
\end{proof}

\begin{cor}
For any partition $\bb N:=\bigsqcup_{i=1}^n C_i$, there is some $i$ such that $C_i$ is both additively and multiplicatively central.
\end{cor}

\begin{proof}
Let $L\subseteq \bb M$ be a multiplicatively minimal left u-ideal.  Take $\alpha\in L$.  Then $\alpha$ is a multiplicatively minimal idempotent.  Take $i$ such that $\alpha \in C_i^*$.  Then $C_i$ is multiplicatively central.  Since $\alpha\in \bb M$, we have that $C_i$ is also additively central.
\end{proof}

Here is a nice combinatorial application of the preceding ideas.  (Something more general appears in \cite[Chapter 10]{book}):

\begin{thm}[Bergelson \cite{Bergelson}]
The equation $x+y=w\cdot z$ is injectively partition regular:  for any partition $\bb N:=\bigsqcup_{i=1}^n C_i$, there is $i$ and distinct $a,b,c,d\in C_i$ such that $a+b=c\cdot d$.  Moreover, $a,b,c,d$ can be chosen arbitrarily large.
\end{thm}

\begin{proof}
Choose $\alpha$ as in the previous corollary, so a multiplicatively minimal idempotent contained in $\bb M$.  Take $i$ such that $\alpha\in C_i^*$.  This $i$ will be as desired.  For notational convenience, set $C:=C_i$.  Since $\alpha\cdot \alpha^*\in C^{**}$, there are arbitrarily large $d\in C$ such that $d\cdot \alpha \in C^*$.  Since $\alpha\in C^*\cap C^*/d$, we have that $C\cap C/d$ is additively central.  Take additively minimal idempotent $\beta \in C^*\cap C^*/d$.  We then have that $\beta+\beta^*\in C^{**}\cap C^{**}/d$.  There are then arbitrarily large $b'\in C\cap C/d$ such that $b'+\beta\in C^*\cap C^*/d$.  Set $b:=b'd$, noting that $b\in C$ and $b\not=d$.  Now note that $d\beta\in C^*\cap (dC^*-b)$, so there are arbitrarily large $a\in C$ such that $a+b\in dC$ (so we can assume that $a\not=b,d$).  Take $c\in C$ such that $a+b=cd$.  Choosing $a$ arbitrarily large forces $c$ arbitrarily large and distinct from $a,b,d$.
\end{proof}

Multiplicatively central sets need not be FS-sets (see \cite[Thm 16.29]{HS}) but we do have the next best thing, a result due to Bergelson and Hindman (see \cite[Thm 3.5]{BH}):

\begin{thm}\label{Berg}
Every multiplicatively piecewise syndetic set is an $\operatorname{FS}_{<\omega}$-set.
\end{thm}

Here, $A\subseteq \bb N$ is an FS$_{<\omega}$-set if there are arbitrarily large finite sets $X$ such that FS$(X)\subseteq A$.  First:

\begin{lem}
Let $\bb F:=\{\alpha \in S^* \ : \ \text{ every }A\in \cU_\alpha \text{ is an }\operatorname{FS}_{<\omega}\text{-set}\}$.  Then $\bb F$ is a nonempty closed multiplicative u-ideal of $S^*$.
\end{lem}

\begin{proof}
$\bb F$ is clearly closed and is nonempty as it contains all additive idempotents.  Suppose that $\alpha \in \bb F$ and $\beta \in S^*$.

First suppose that $\alpha\cdot \beta^*\in A^{**}$ and fix $k\in \bb N$.  Since $\alpha\in A_\beta^*$, we see that $A_\beta$ is an FS$_{<\omega}$-set, say FS$(\{x_1,\ldots,x_k\})\subseteq A_\beta$.  This is a finitary statement, whence there is $y\in \bb N$ such that FS$(\{x_1,\ldots,x_k\})\cdot y\subseteq A$, that is, FS$(\{x_1y,\ldots,x_ky\})\subseteq A$.

Now suppose that $\beta\cdot \alpha^*\in A^{**}$.  Then there is $n\in \bb N$ such that $n\alpha\in A^*$, so $A/n$ is FS$_{<\omega}$, whence so is $A$.
\end{proof}

\begin{proof}[Proof of Theorem \ref{Berg}]
Suppose that $A\subseteq \bb N$ is multiplicatively piecewise syndetic.  Let $\alpha \in A^*$ be a multiplicatively minimal idempotent element.  By the previous lemma, $\alpha \in \bb F$, whence $A$ is FS$_{<\omega}$.
\end{proof}

\section{Combinatorial descriptions of central sets}

In this section, we give a description of central sets that is purely combinatorial.  We split this task up into two parts.

\subsection{Part 1:  FP-trees}

\emph{In this section (and this section only), we follow usual set-theoretic convention and view $n\in \bb N$ as the ordinal $n=\{0,1,\ldots,n-1\}$.}
\begin{defn}
If $X$ is a set, a \textbf{tree in $X$} is a set $T\subseteq X^{<\omega}$ closed under initial segments.

Suppose that $T$ is a tree in $X$.
\begin{enumerate}
\item For $f\in T$ with $\operatorname{dom}(f)=n$ and $x\in X$, we set $f^\frown x:=f\cup \{(n,x)\}$. 
\item For $f\in T$, we set $B_f:=\{x\in X \ : \ f^\frown x\in T\}$.
\item We say that $T$ is \textbf{pruned} if $B_f\not=\emptyset$ for all $f\in T$.
\end{enumerate}
\end{defn}

\begin{defn}
Suppose that $S$ is a semigroup and $T$ is a tree in $S$.  
\begin{enumerate}
\item For $f\in T$, we set $P_f:=FP(\langle f(t)\rangle_{t\in \operatorname{dom}(f)})$.
\item For $f,g\in T$ with $f\subsetneq g$, we set $$P_{g-f}:=FP(\langle  g(t)\rangle_{t\in \operatorname{dom}(g), t\geq \operatorname{dom}(f)}).$$
\item We say that $T$ is an \textbf{FP-tree} if, for each $f\in T$, we have
$$B_f:=\bigcup_{f\subsetneq g}P_{g-f}.$$
\end{enumerate}
\end{defn}

Note that the inclusion $(\subseteq)$ in the previous display always holds.

Suppose that $A\subseteq S$ is such that there is a pruned FP-tree in $A$.  Then $A$ is an FP-set.  Indeed, if $\sigma\in A^\omega$ is an infinite branch in $T$ (meaning that $\sigma|n\in T$ for all $n\in \omega$), then $FP(\langle \sigma(n)\rangle_{n=1}^\infty)\subseteq B_\emptyset\subseteq A$.  Surprisingly, the converse holds:

\begin{thm}\label{FPtree}
$A\subseteq S$ is an FP-set if and only if there is a pruned FP-tree in $A$.  In fact, if $\alpha\in A^*$ is idempotent, then there is an FP-tree $T$ in $A$ such that $\alpha\in B_f^*$ for all $f\in T$. 
\end{thm}

\begin{proof}
Suppose that $\alpha\in A^*$ is idempotent.  We construct $T$ level by level by recursion so that $P_f \subseteq \aA$ for all $f\in T$.  Clearly $T_0:=\{\emptyset\}$.  We set $T_1:=\aA$.  Now suppose that $f\in T_n$.  We then set $B_f:=\{x\in \aA \ : \ P_f\cdot x\subseteq \aA\}$, noting that $\alpha\in B_f^*$ (since $P_f\subseteq \aA$).  It will be useful to observe that the construction ensures that, for $f,g\in T$ with $f\subseteq g$, we have $B_g\subseteq B_f$.  

We now verify that $T$ is an FP-tree.  Suppose that $f\subsetneq g$ and $x\in P_{g-f}$.  Write $x=g(t_1)\cdots g(t_n)$ with $t_1\geq \operatorname{dom}(f)$.  Let $h$ be the restriction of $g$ with domain $t_n$.  If $n=1$, then $x=g(t_1)\in B_h\subseteq B_f$, as desired.  Otherwise, set $w:=g(t_1)\cdots g(t_{n-1})\in P_h$.  We need to show that $x=wg(t_n)\in B_f$.  Since $g(t_n)\in B_h$, we have that $x\in \aA$ by definition of $B_h$.  To see that $P_f\cdot x\subseteq \aA$, note that $P_fw\subseteq P_h$, so $P_f\cdot x\subseteq P_h\cdot g(t_n)\subseteq \aA$, again by the definition of $B_h$.    
\end{proof}

The desired combinatorial characterization of central sets arises from strengthening the notion of FP-tree.

\begin{defn}[Temporary]
We call $\cal A\subseteq \cal P(S)$ \textbf{robust} if there is minimal $\alpha\in \bigcap_{A\in \cal A}A^*$.  We call an FP-tree $T$ \textbf{robust} if $\{B_f \ : \ f\in T\}$ is robust.
\end{defn}

\begin{cor}
If $A$ is central, then there is a robust FP-tree in $A$.
\end{cor}

We aim to show that the converse holds.  First:

\begin{defn}
Suppose that $\cal C:=(C_i)_{i\in I}$ is a family of subsets of $A$ with the finite intersection property.  We say that $\cal C$ is \textbf{good} if:  for every $i\in I$ and every $x\in C_i$, there is $j\in I$ such that $x\cdot C_j\subseteq C_i$.
\end{defn}

\begin{lem}
Suppose that $\cal C$ is a good family of subsets of $S$.  Then $\bigcap_{i\in I}C_i^*$ is a nonempty closed $u$-subsemigroup of $S^*$.
\end{lem}

\begin{proof}
Set $M:=\bigcap_{i\in I}C_i^*$.  $M$ is clearly closed and is nonempty by the finite intersection property.  To see that $M$ is a u-subsemigroup, suppose $\alpha,\beta\in M$.  It suffices to show that $\alpha\cdot \beta^*\in C_i^{**}$ for all $i\in I$ (for then any $\gamma\sim \alpha\cdot \beta^*$ belongs to $M$).  Fix $x\in C_i$ and take $j$ such that $x\cdot C_j\subseteq C_i$.  We then have that $x\beta\in C_i^*$.  It follows that $C_i\subseteq (C_i)_\beta$, so $\alpha\in (C_i)_\beta^*$, that is, $\alpha\beta^*\in C_i^{**}$, as desired.
\end{proof}             

\begin{cor}
Suppose that there is a robust good family of subsets of $A$.  Then $A$ is central.
\end{cor}

\begin{proof}
Let $(C_i)_{i\in I}$ be a robust good family of subsetsof $A$.  Then there is a minimal left $u$-ideal $L$ such that $L\cap \bigcap_{i\in I}C_i^*\not=\emptyset$.  Since $L\cap \bigcap_{i\in I}C_i^*$ is a closed $u$-subsemigroup of $S^*$, it contains an idempotent element $\alpha$; since $\alpha\in A^*$, we see that $A$ is central.
\end{proof}

To bridge the gap between FP-trees and good families, we make a new definition:

\begin{defn}
Suppose that $T$ is a tree in $S$.  We say that $T$ is a \textbf{$*$-tree} if:  for every $f\in T$ and every $x\in B_f$, we have $x\cdot B_{f^\frown x}\subseteq B_f$.
\end{defn}
In other words, $T$ is a $*$-tree if:  for all $f\in T$ and $x,y\in S$, if $f^\frown x^\frown y\in T$, then $f^\frown xy\in T$.  The following lemma is routine and does not need any nonstandard methods.  See \cite[Lemma 14.23.1 and Theorem 14.25]{HS} for proofs.

\begin{lem}

\

\begin{enumerate}
\item Every FP-tree is a $*$-tree.
\item Suppose that $T$ is a $*$-tree and $C_F:=\bigcap_{f\in F}B_f$.  Then $(C_F)$ is a good family.  If $T$ is robust, then the family is robust.
\end{enumerate}
\end{lem}

We can now summarize:

\begin{thm}
For $A\subseteq S$, the following are equivalent:
\begin{enumerate}
\item $A$ is central.
\item There is a robust FP-tree in $A$.
\item There is a robust $*$-tree in $A$.
\item There is a robust good family of subsets of $A$.
\end{enumerate}
\end{thm}

\subsection{Part 2:  Collectionwise piecewise syndetic families}

The issue with the previous theorem is that it is only provides a quasi-combinatorial characterization of central set as it uses the notion of robustness, which is defined in terms of ultrafilters.  The goal of this subsection is to give a combinatorial characterization of robustness.  The basic idea is that since piecewise syndeticity is the same as containing a minimal element, robustness will be equivalent to some form of uniform piecewise syndeticity.  Here is the standard definition that will turn out to be equivalent to robustness:

\begin{defn}
We say that $\cal A\subseteq \cal P(S)$ is \textbf{collectionwise piecewise syndetic} (cwpws) if there are functions $G:\cal P_f(\cal A)\to \cal P_f(S)$ and $x:\cal P_f(\cal A)\times \cal P_f(S)\to S$ such that, for all $F\in \cal P_f(S)$ and all $\cal F,\cal H\in \cal P_f(\cal A)$ with $\cal F\subseteq \cal H$, we have  
$$F\cdot x(\cal H,F)\subseteq G(\cal F)^{-1}(\bigcap \cal F).$$
\end{defn}

Before giving some nonstandard reformulations of cwpws, we remind the reader that, given a set $X$, a \emph{hyperfinite approximation of $X$} is a hyperfinite set $H$ such that $X\subseteq H\subseteq X^*$.  Given enough saturation, every set has a hyperfinite approximation.

\begin{thm}\label{nscwpws}
For a semigroup $S$ and $\cal A\subseteq \cal P(S)$, the following are equivalent:
\begin{enumerate}
\item $\cal A$ is cwpws.
\item For any hyperfinite approximation $H$ of $S$, there is $\alpha \in S^*$ such that, for any finite subset $\cal F$ of $\cal A$, we have $H\cdot \alpha\subseteq S^{-1}(\bigcap \cal F)^*$.
\item For any hyperfinite approximation $H$ of $S$, there is $\alpha\in S^*$ and $G$ as above such that, for any finite subset $\cal F$ of $\cal A$, we have $H\cdot \alpha\subseteq G(\cal F)^{-1}(\bigcap \cal F)^*$.
\item There is a hyperfinite approximation $H$ of $S$,  $\alpha\in S^*$, and $G$ as above such that, for any finite subset $\cal F$ of $\cal A$, we have $H\cdot \alpha\subseteq G(\cal F)^{-1}(\bigcap \cal F)^*$.
\item There is $\alpha\in S^*$ and $G$ as above such that, for any finite subset $\cal F$ of $\cal A$, we have $S\cdot \alpha\subseteq G(\cal F)^{-1}(\bigcap \cal F)^*$.
\end{enumerate}
\end{thm}

\begin{proof}
(1) implies (2):  Let $G$ and $x$ witness that $\cal A$ is cwpws.  Let $H$ be any hyperfinite approximation of $S$ and let $\cal H$ be any hyperfinite approximation of $\cal A$.  Then by transfer, for any finite $\cal F$ contained in $\cal A$, we have 
$$H\cdot x(\cal H,H)\subseteq G(\cal F)^{-1}(\bigcap \cal F)^*\subseteq S^{-1}(\bigcap \cal F)^*.$$

(2) implies (3) follows immediately from saturation.

(3) implies (4) and (4) implies (5) are trivial.

(5) implies (1):  Given finite $F$ contained in $S$ and finite $\cal H$ contained in $\cal A$, there are only finite many $\cal F$ contained in $\cal H$, so apply transfer to the statement ``there is $x\in S^*$ such that, for all $\cal F$ contained in $\cal H$, $F\cdot x\subseteq G(\cal F)^{-1}(\bigcap \cal F)^*$'' to get $x(\cal H,F)$.
\end{proof}

Although the following corollary can be deduced with some effort from the standard definition, it is an immediate consequence of the previous theorem.

\begin{cor}
$\cal A\subseteq \cal P(S)$ is cwpws if and only if the closure of $\cal A$ under finite intersections is cwpws.
\end{cor}

\begin{remark}
Consider the case of $(\bb N,+)$.  $A\subseteq \bb N$ is piecewise syndetic if and only if there is a hyperfinite interval $I$ such that $A^*\cap I$ has only finite gaps.  Suppose, for simplicity, that $\cal A\subseteq \cal P(\bb N)$ is closed under finite intersections.  The above theorem shows that $\cal A$ is collectionwise piecewise syndetic if and only if there is a hyperfinite interval $I$ such that $A^*\cap I$ has only finite gaps for every $A\in \cal A$.  The cleanliness of the previous statement is why we find the nonstandard description of the notion of cwpws family so natural.
\end{remark}

Here is the main result of this subsection, which completes the combinatorial description of central sets:

\begin{thm}
For $\cal A\subseteq \cal P(S)$, we have that $\cal A$ is cwpws iff $\cal A$  is robust
\end{thm}

\begin{proof}
First suppose that $\cal A$ is cwpws and take $\alpha$ and $G$ as in condition (5) of Theorem \ref{nscwpws}.  By transfer, we have $S^*\cdot \alpha^*\subseteq G(\cal F)^{-1}(\bigcap \cal F)^{**}$.  Let $\gamma \in\cal L(S^*\cdot \alpha^*)$ be minimal.  For any $\cal F$, let $t(\cal F)\in G(\cal F)$ be such that $\gamma\in t(\cal F)^{-1}(\bigcap \cal F)^{*}$.  and set $E(\cal F):=\{t(\cal H) \ : \ \cal F\subseteq \cal H\}$.  Note that, in particular, for any $A\in \cal A$, that $E(\{A\})\subseteq A_\gamma$.  The family of $E(\cal F)$'s has the finite intersection property, whence there is $\beta \in \bigcap_{\cal F}E(\cal F)^*$.  Then, for $A\in \cal A$, we have $\beta\in A_\gamma^*$, that is, $\beta\cdot \gamma^*\in A^{**}$, so $A\in \cal U_{\beta\cdot \gamma^*}$.  It remains to note that $\beta \cdot \gamma^*$ is minimal.  

Conversely, suppose that $\alpha\in S^*$ is minimal such that $\cal A\subseteq \cal U_\alpha$.  We claim that $\alpha$ witnesses the truth of (5) in Theorem \ref{nscwpws}.  Fix $\cal F$ and, for notational simplicity, set $A:=\bigcap \cal F$.  Since $\alpha$ is minimal, $A_\alpha$ is syndetic, whence there is finite $G(\cal F)\subseteq S$ such that $S^*\subseteq G(\cal F)^{-1}A_\alpha^*$.  It follows that $S^*\alpha^*\subseteq G(\cal F)^{-1}A^{**}$, whence, by transfer, $S\alpha\subseteq G(\cal F)^{-1}A^*$.
\end{proof}


\section{The Central sets theorem}

In this section, we state and prove the Central Sets Theorem, which is arguably the most important result about central sets in applications.

\subsection{The Central Sets Theorem:  statement and consquences}
We first set up some important notation and definitions.
\begin{defn}
Given $m\in \bb N$, $a\in S^{m+1}$, $f\in S^\bb N$, and $t\in \bb N^{[m]}$, we set
$$x(m,a,t,f):=a(1)f(t_1)a(2)f(t_2)\cdots a(m)f(t_m)a(m+1).$$
\end{defn}

\begin{defn}
$A\subseteq S$ is a \textbf{C-set} if there are:
\begin{enumerate}
\item $m:\cal P_f(S^\bb N)\to \bb N$
\item $a\in \prod_{F\in \cal P_f(S^\bb N)}S^{m(F)+1}$
\item  $t\in \prod_{F\in \cal P_f(S^\bb N)}\bb N^{[{m(F)}]}$
\end{enumerate}
satisfying:
\begin{enumerate}
\item[(a)] $F\subsetneq G$ implies $t(F)(m(F))<t(G)(1)$, and
\item[(b)] for all $G_1\subsetneq \cdots\subsetneq G_n$ and $g_i\in G_i$, we have
$$\prod_{i=1}^n x(m(G_i),a(G_i),t(G_i),g_i)\in A.$$
\end{enumerate}
\end{defn}

Our goal is to prove:

\begin{thm}[Central Sets Theorem]\label{centralsettheorem}
Every central set is a C-set.
\end{thm}

\begin{remark}
The converse to the Central Sets Theorem is false; see \cite[Thm 14.18]{HS} for a concrete counterexample.
\end{remark}

We will prove the Central Sets theorem in the next section.  The version of the theorem presented here is the strongest known version of the theorem, which has undergone several improvements since its original version, due to Furstenberg:

\begin{thm}[Furstenberg's Central Sets Theorem]
Suppose that $A\subseteq \bb N$ is central and $\langle y^1_{n}\rangle_{n=1}^\infty,\ldots,\langle y^k_{n}\rangle_{n=1}^\infty$ are sequences in $\bb Z$.  Then there is a sequence $\langle a_n\rangle_{n=1}^\infty$ from $\bb N$ and an increasing sequence $\langle H_n\rangle_{n=1}^\infty$ from $\cal P_f(\bb N)$ (meaning that $\max H_n<\min H_{n+1}$) such that, for all $i=1,\ldots,k$, we have
$$FS\left(\left\langle a_n+\sum_{t\in H_n}y^i_t\right\rangle_{n=1}^\infty\right)\subseteq A.$$
\end{thm} 

For a discussion of how to derive Furstenberg's Central Set Theorem from Theorem \ref{centralsettheorem}, see \cite{Hindman}.

Furstenberg used the Central Sets Theorem to establish that any (kernel) partition regular system of equations over $\bb Q$ must have a solution in any central set.  A later application of the Central Sets Theorem showed that a system of equations over $\bb Q$ is image partition regular if and only if the column space of the matrix for the equation meets every central set. 

The combinatorial applications of the Central Sets Theorem are quite numerous and we suggest that the reader consult \cite{Hindman} and \cite{HS} for more information.  Since these applications involve straightforward (but nontrivial) standard reasoning using the Central Sets Theorem, we shall say no more about them here.

\subsection{The proof of the Central Set Theorem}

To prove the Central Sets Theorem, we need an auxiliary notion:

\begin{defn}
$A\subseteq S$ is a \textbf{J-set} if:  for every $F\in \cal P_f(S^\bb N)$, there is $m\in \bb N$, $a\in S^{m+1}$, and $t\in \bb N^{[m]}$ such that $x(m,a,t,f)\in A$ for all $f\in F$.
\end{defn}

We note an easy observation about J-sets:

\begin{lem}
Suppose that $A\subseteq S$ is a J-set, $F\in \cal P_f(S^\bb N)$, and $k\in \bb N$.  Then there are $m,a,t$ with $t(1)>k$ such that $x(m,a,t,f)\in A$ for all $f\in F$.
\end{lem}

\begin{proof}
Apply the definition of J-set to $G:=\{g_f \ : \ f\in F\}$ with $g_f(n):=f(k+n)$.
\end{proof}

There is an obvious nonstandard formulation of being a J-set, but we have not found it too useful thus far:

\begin{lem}
$A\subseteq S$ is a J-set if and only if there is $M\in \bb N^*$, $a\in (S^*)^{M+1}$, and $t\in (\bb N^*)^{[M]}$ such that, for all $f\in S^\bb N$, we have
$$x(M,a,t,f)\in A^*.$$
\end{lem}

\begin{defn}
We call $\alpha\in S^*$ a \textbf{J-element} (resp. \textbf{C-element}) if every $A\in \cU_\alpha$ is a J-set (resp. C-set).
\end{defn}

One proves the Central Set Theorems in two steps:

Step 1:  Show that every minimal element is a J-element.  

Step 2:  Show that every idempotent J-element is a C-element.  

\begin{remarks}

\

\begin{enumerate}
\item It follows from Step 1 that that every piecewise syndetic set is a J-set.
\item The converse of the statement in Step 2 is also true, but since its proof is much more involved, we will not prove it here.  Note that this converse implies that the collection of C-sets is also partition regular.
\end{enumerate}
\end{remarks}

We start with Step 1.  There is a direct proof of Step 1 that uses more facts about minimal ideals than we would like to present here (see \cite[Thm 2.11]{johnson}).  We prefer the following strategy towards establishing Step 1:

Step 1a:  Show that there is a J-element.

Step 1b:  Show that the set of J-elements is a (nonempty by Step 1a) u-ideal of $S^*$.

To prove Step 1a, we first prove:

\begin{thm}\label{Jreg}
The family of J-sets is partition regular. 
\end{thm}

The proof of Theorem \ref{Jreg} that we present here is completely standard but we choose to give it as:  (a) it is very clever; and (b) the other proofs in the literature that we have seen have chosen to focus more on the details then the basic ideas.  The proof uses the Hales-Jewett theorem, which we now describe.  

Suppose that $A$ is a finite nonempty set (our \emph{alphabet}).  A \emph{word on $A$} is simply an element of $A^n$ for some $n\in \bb N$; we refer to $n$ as the \emph{length of the word}.  A \emph{variable word on $A$} is a word on the alphabet $A\cup \{\star\}$, where $\star$ is a new element not belonging to $A$, such that $\star$ actually occurs in $w$.  Given a variable word $w(\star)$ and $a\in A$, we set $w(a)$ to be the word on $A$ obtained by replacing each occurrence of $\star$ by $a$.  Finally, given a variable word $w(\star)$, the set $\{w(a) \ : \ a\in A\}$ is referred to as a \emph{combinatorial line}.

\begin{fact}[Hales-Jewett Theorem]
For every $k,c\in \bbN$, there is $N=N(k,c)\in \bbN$ such that, for every set $A$ of size $k$ and every coloring of words on $A$ of length $N$ using $c$ colors, there is a length $N$ variable word $w(\star)$ on $A$ such that the combinatorial line $\{w(a) \ : \ a\in A\}$ is monochromatic.
\end{fact}

\begin{remark}
The previous theorem is actually known as the finitary Hales-Jewett Theorem, which can be derived, using a familiar compactness argument, from the infinitary Hales-Jewett Theorem.  For a nonstandard proof of the latter fact, using many of the ideas present in this note, see \cite[Chapter 8, Section 2]{book}.  We should also mention that it is quite easy to derive the infinitary Hales-Jewett Theorem from the Central Sets Theorem.  For this reason, the direct approach to proving Step 1 is preferable in that it avoids any circular reasoning.
\end{remark}

\begin{proof}[Proof of Theorem \ref{Jreg}]
Suppose that $A,B\subseteq S$ are such that $A\cup B$ is a J-set but $A$ is not a J-set.  We show that $B$ is a J-set.  Fix $F\in \cal P_f(S^\bb N)$.  We find $m,a\in S^{m+1}$, and $t\in \bb N^{[m]}$ such that $x(m,a,t,f)\in B$ for all $f\in F$.

Let $G\in \cal P_f(S^\bb N)$ witness that $A$ is not a J-set.  Set $H:=F\cup G$ and write $H:=\{h_1,\ldots,h_k\}$.  Let $N:=N(k,2)$ be as in the Hales-Jewett theorem.  Below we will define, for $w\in \{1,\ldots,k\}^N$, elements $g_w\in S^\bb N$.  Since $A\cup B$ is a J-set, there are $p,b\in S^{p+1}$, and $s\in \bb N^{[p]}$ such that $x(p,b,s,g_w)\in A\cup B$ for all $w\in \{1,\ldots,k\}^N$.

Define a coloring on elements of $\{1,\ldots,k\}^N$ by setting $c(w)$ red if $x(p,b,s,g_w)\in A$ and $c(w)$ blue otherwise.  By the choice of $N$, there is a variable word $w(\star)$ on $\{1,\ldots,k\}$ of length $N$ such that the combinatorial line $\{w(i) \ : \ i=1,\ldots,k\}$ is monochromatic.

\textbf{Claim:}  There are $m,a\in S^{m+1}$, and $t\in \bb N^{[m]}$ such that, for each $i=1,\ldots,k$, we have
$$x(m,a,t,h_i)=x(p,b,s,g_{w(i)}).$$

Taking the claim for granted, we see that the monochromatic combinatorial line cannot have color red, else we contradict the choice of $G$.  It follows that the monochromatic combinatorial line has color blue, which implies, in particular, that $x(m,a,t,f)\in B$ for all $f\in F$, as desired.

It remains to describe the elements $g_w$ and verify the claim for these elements.  Fix arbitrary $d\in S$ arbitrarily.  (We will soon see that $d$ is merely a ``space-filler''.)  For $w=(w_1,\ldots,w_N)\in \{1,\ldots,k\}^N$, we set
$$g_w(l):=\prod_{i=1}^N d\cdot h_{w_i}(Nl+i).$$  To verify the claim, let $v_1<\cdots<v_r$ enumerate (in order) the appearances of $\star$ in $w(\star)$.  For $*=1,\ldots,k$, we have $p\cdot r$ many appearances of $h_{\star}$ in $x(p,b,s,g_{w(\star)})$, with inputs
$$Ns_1+v_1<\cdots<Ns_1+v_r<\cdots<Ns_p+v_1<\cdots <Ns_p+v_r.$$  (There may be other, incidental, appearances of a given element of $H$, but we want something uniform in $\star$.)  We set $m:=p\cdot r$ and let the above sequence be $t$.   The ``padding'' in the aforementioned product is then our desired $a$.  Note that $d$ is used in case of consecutive appearances of $\star$.  We leave it to the reader to write down precise formulae if they desire; otherwise, they can consult \cite[Thm 2.5]{johnson}.
\end{proof}

\begin{cor}
There is a J-element in $S^*$.  In fact, $A$ is a J-set if and only if $A^*$ contains a J-element.
\end{cor}

\begin{proof}
Let $\cal J:=\{A\subseteq S \ : \ A^c \text{ is not a J-set}\}$.  By the partition regularity of the collection of J-sets, we see that $\cal J$ has the finite intersection property, whence there is $\alpha \in \bigcap_{A\in \cal J} A^*$.  It follows that $\alpha$ is a J-element.  The moreover follows from the fact that every element of $\cal J$ meets every J-set.
\end{proof}

We now deal with Step 1b:

\begin{prop}
The set of J-elements is a u-ideal of $S^*$.
\end{prop}

\begin{proof}
Consider $\alpha,\beta\in S^*$ with $\alpha$ a J-element.

$\alpha\beta^*$ is a J-element:  Suppose that $\alpha\beta^{*}\in A^{**}$.  Then $\alpha\in A_\beta^*$, so $A_\beta$ is a J-set.  Fix $F\subseteq \cal P_f(S^\bb N)$ and take $m,a,t$ such that $x(m,a,t,f)\in A_\beta$ for all $f\in F$, that is, $x(m,a,t,f)\beta \in A^*$ for all $f\in F$, whence there is $s\in S$ such that $x(m,a,t,f)s\in A$ for all $f\in F$.  Taking $a'\in S^{m+1}$ to agree with $a$ except that $a'(m+1):=a(m+1)\cdot s$, we see that $x(m,a',t,f)\in A$ for all $f\in F$.

$\beta\alpha^*$ is a J-element:  This is much easier.  Suppose $\beta\alpha^*\in A^{**}$.  Then $\beta\in A_\alpha^*$, so, by transfer, there is $s\in A_\alpha$, that is, $s^{-1}A$ is a J-set.  It follows easily that $A$ is also a J-set.
%
\end{proof}

This completes the proof of Step 1.  We now move on to Step 2.

\begin{thm}
Suppose that $\alpha\in S^*$ is an idempotent J-element.  Then $\alpha$ is a C-element.
\end{thm}

\begin{proof}
Suppose $\alpha\in A^*$; we need to show that $A$ is a C-set, which we accomplish by constructing, by recursion on the size of $F$, functions $m(F)$, $a(F)$, and $t(F)$ satisfying:
\begin{enumerate}
\item[(i)] for all $\emptyset\not=G\subsetneq F$, we have $t(G)(m(G))<t(F)(1)$; and
\item[(ii)] for all $G_1\subsetneq\cdots\subsetneq G_n=F$ and $g_i\in G_i$, we have
$$\prod_{i=1}^n x(m(G_i),\alpha(G_i),t(G_i),g_i)\in \aA.$$
\end{enumerate} 
For $F=\{f\}$, we simply take $m(F)$, $a(F)$, and $t(F)$ witnessing that $\aA$ is a J-set for $F$, which follows from the fact that $\alpha\in \aA^*$ and that $\alpha$ is a J-element.

Now suppose that $m(G)$, $a(G)$, and $t(G)$ have been defined for all proper subsets of $F$ satisfying (i) and (ii).  Let $k=\max_{\emptyset\not=G\subsetneq F}t(G)(m(G))$ and let 
$$M:=\left\{\prod_{i=1}^n x(m(G_i),a(G_i),t(G_i),g_i) \ : \ \emptyset \not=G_1\subsetneq \cdots \subsetneq G_n\subsetneq F, \ g_i\in G_i\right\}.$$  Since $M$ is a finite subset of $\aA$, we have that $M\cdot \alpha\subseteq \aA^*$.  Setting $B:=\{x\in \aA \ : \ M\cdot x\subseteq \aA\}$, we have that $\alpha\in B^*$, whence $B$ is a J-set.  We then let $m(F)$, $a(F)$, and $t(F)>k$ be as in the definition of J-set for $B$ corresponding to $F$.  It is clear that items (i) and (ii) of the recursion are still satisfied.  
\end{proof}

This completes the proof of the Central Sets Theorem.

\section{The Dynamic Definition}

In this section, we give the dynamic definition of central set and prove the equivalence with the earlier version.
\subsection{Dynamic preliminaries}

We start with some definitions.

\begin{defn}
A \textbf{dynamical system} is a pair $(X,\langle T_s\rangle_{s\in S})$ such that:
\begin{enumerate}
\item $X$ is a compact space;
\item $S$ is a semigroup;
\item for $s\in S$, $T_s:X\to X$ is continuous;
\item for $s,t\in T$, we have $T_{st}=T_s\circ T_t$.
\end{enumerate}
\end{defn}

Given a dynamical system as above, $s\in S$, and $x\in X$, we sometimes write $s\cdot x$ instead of $T_s(x)$.  We also let $T:S\to X^X$ denote the function $T(s)(x):=T_s(x)$.

\begin{remark}
Note that the natural left action of $S$ on $\beta S$ yields a dynamical system in the above sense.  In this way, closed subsystems of $\beta S$ correspond to left ideals.  In topological dynamics, studying minimal closed subsystems is natural as they correspond to the irreducible objects.  Minimal closed subsystems of $\beta S$ thus correspond to minimal left ideals.  It might have seemed strange at first to be so concerned with minimal left ideals, but we see now that they are a very natural object of study from the dynamic point of view.
\end{remark}

The following lemma is standard and easy:

\begin{lem}\label{Bernoulli}
Suppose that $Q$ is a semigroup with subsemigroup $S$.  Let $X:=2^Q$ (with the product topology).  For $s\in S$, set $T_s(f)(x):=f(xs)$.  Then $(X,\langle T_s\rangle_{s\in S})$ is a dynamical system.
\end{lem}

Until further notice, fix a dynamical system $(X,\langle T_s\rangle_{s\in S})$.  Given $x\in X$ and a subset $U$ of $X$, we consider the \textbf{return set}
$$R(T,x,U):=R(x,U):=\{s\in S \ : \ T_s(x)\in U\}.$$

A focal point of topological dynamics is the study of various properties of return sets.  Here is a very natural definition along these lines:

\begin{defn}
We say that $x\in X$ is \textbf{uniformly recurrent} if, for every neighborhood $U$ of $x$, we have that $R(x,U)$ is syndetic.
\end{defn}

The previous nomenclature is easiest to digest when considering dynamical systems over $(\bb N,+)$, which is tantamount to studying the iterates of a single continuous transformation $T$.  In this case, $x\in X$ is uniformly recurrent if, for any neighborhood $U$ of $x$, there is $m\in \bb N$ such that, for any $y$ in the orbit of $x$, we have that $y$ returns to $U$ within $m$ iterates of $T$.

The study of uniformly recurrent points is also intimately tied up with minimal dynamical systems referred to above.  Indeed, one can show that every point in a minimal dynamical system is uniformly recurrent and, conversely, the orbit closure of a uniformly recurrent point is a minimal system.  (See \cite[Theorems 1.15 and 1.17]{Furst}.)

Here is the other preliminary definition we need:

\begin{defn}
We say that $x,y\in X$ are \textbf{proximal} if there is a net $(s_i)_{i\in I}$ such that $\lim_{i\in I}T_{s_i}(x)=\lim_{i\in I}T_{s_i}(y)$.
\end{defn}

We now come to the dynamic definition of central set, which we temporarily give a different name until we show that it coincides with our earlier notion of central set.

\begin{defn}
If $S$ is a semigroup, then $A\subseteq S$ is \textbf{dynamically central} if there is a dynamical system $(X,\langle T_s\rangle_{s\in S})$ and points $x,y\in X$ such that:
\begin{enumerate}
\item $x$ and $y$ are proximal;
\item $y$ is uniformly recurrent;
\item $A=R(x,U)$.
\end{enumerate}
\end{defn}


\subsection{The equivalence between central sets and dynamically central sets}  We now proceed to show that the notions ``central'' and ``dynamically central'' coincide.  As before, we fix a dynamical system $(X,\langle T_s\rangle_{s\in S})$.  Recall that every element $x\in X^*$ has a unique standard part, denoted $\st(x)\in X$, with the property that $x\approx \st(x)$ (meaning:  whenever $U$ is a neighborhood of $\st(x)$, then $x\in U^*$).  By iterated applications of transfer, this fact remains true for every $x\in X^{(\infty)*}$.  This allows us to define a function $\tT:S^*\to X^X$ by setting $$\tT(\alpha)(x):=\tT_\alpha(x):=\st(\alpha\cdot x)=\st(T_\alpha(x)).$$  Note that $\tT$ extends $T$.

\begin{lem}
For $\alpha,\beta\in S^{(\infty)*}$, we have $\alpha\sim \beta$ implies $\tT_\alpha=\tT_\beta$.
\end{lem}

\begin{proof}
For ease of notation, suppose $\alpha,\beta \in S^*$.  If $z=\st(\alpha\cdot x)$ and $U$ is a neighborhood of $z$, setting $A:=\{s\in S \ : \ s\cdot x\in U\}$, we have $\alpha\in A^*$, whence $\beta\in A^*$, that is $\beta\cdot x\in U^*$.  Since $U$ was an arbitrary neighborhood of $z$, we see that $z=\st(\beta\cdot x)$.
\end{proof}

As with the  map $\pi$, $\tT$ need not be a semigroup homomorphism.  However:
\begin{prop}
For any $\alpha,\beta\in S^*$, we have $\tT_{\alpha\cdot \beta^*}=\tT_\alpha\circ \tT_\beta$.
\end{prop}

\begin{proof}
Fix $x\in X$.  Since $(\alpha\beta^*)\cdot x=\alpha\cdot (\beta^*\cdot x)$, we need to show
$$\st(\alpha\cdot (\beta^*\cdot x))=\st(\alpha\cdot(\st(\beta\cdot x))).$$  Set $y:=\st(\beta\cdot x)$ and $z:=\st(\alpha\cdot y)$.   We need to show that $z=\st(\alpha\cdot (\beta^*\cdot x))$.  Fix a neighborhood $U$ of $z$; we must show that $\alpha\cdot (\beta^*\cdot x)\in U^{**}$.  Since each $T_s$ is continuous, the statement ``for all $s\in S$ and all open sets $V$, if $s\cdot y\in V$, then $s\cdot (\beta\cdot x)\in V^*$'' is a true statement.  By transfer, we have that ``for all $\gamma\in S^*$ and all internally open sets $V$, if $\gamma\cdot y\in V$, then $\gamma\cdot (\beta^*\cdot x)\in V^*$'' is also true.  We finish by setting $\gamma=\alpha$ and $V=U^*$. 
\end{proof}

We next give the nonstandard reformulation of proximality:
\begin{lem}
$x,y\in X$ are proximal if and only if there is $\alpha\in S^*$ such that $\tT_\alpha(x)=\tT_\alpha(y)$.
\end{lem}

\begin{proof}
First suppose that $x$ and $y$ are proximal, say $\lim_{i\in I}T_{s_i}(x)=\lim T_{s_i}(y)$.  Fix $i\in I^*$ with $i>I$.  Then $T_{s_i}(x)=T_{s_i}(y)$.

Conversely, suppose that $\tT_\alpha(x)=\tT_\alpha(y)$ and let $z$ be the common standard part.  For each neighborhood $U$ of $\tT_\alpha(x)$, we have, by transfer, some $s_U\in S$ such that $s_U\cdot x,s_U\cdot y\in U$.  It follows that $\lim_U s_U\cdot x=\lim_U s_U\cdot y=z$.
\end{proof}

For $x,y\in X$, let $I(x,y):=\{\alpha\in S^* \ : \ \tT_\alpha(x)=\tT_\alpha(y)\}.$  Thus, $x$ and $y$ are proximal if and only if $I(x,y)\not=\emptyset$.

\begin{lem}
$I(x,y)$ is a left u-ideal of $S^*$.
\end{lem}

\begin{proof}
Suppose that $\beta\in I(x,y)$ and $\alpha\in S^*$.  Then
$$\tT_{\alpha\cdot \beta^*}(x)=\tT_\alpha(\tT_{\beta}(x))=\tT_\alpha(\tT_{\beta}(y))=\tT_{\alpha\cdot \beta^*}(y).$$  Thus, if $\gamma\sim \alpha\cdot \beta^*$, we have $\tT_\gamma(x)=\tT_\gamma(y)$, whence $\gamma\in I(x,y)$.
\end{proof}

In the following proof, we will need one more fact about $K(\beta S)$, namely, for every minimal left ideal $L$ of $\beta S$ and every $\cU\in L$, there is an idempotent $\cal V\in L$ such that $\cU\cdot \cal V=\cU$.  (See \cite[Thm 1.61]{HS}.)

\begin{thm}
Suppose that $y\in X$ and $L$ is a minimal left u-ideal of $S^*$.  Then the following are equivalent:
\begin{enumerate}
\item $y$ is uniformly recurrent.
\item There is $\alpha\in L$ such that $\tT_\alpha(y)=y$.
\item There is idempotent $\alpha\in L$ such that $\tT_\alpha(y)=y$.
\item There is idempotent $\alpha\in L$ and $x\in X$ such that $\tT_\alpha(x)=y$.
\end{enumerate}
\end{thm}

\begin{proof}
(1) implies (2):  Fix $\beta \in L$.  Let $U$ be an internally open neighborhood of $y$ contained in the monad of $y$, that is, every element of $U$ is infinitely close to $y$.  Since $y$ is uniformly recurrent, by transfer, $R(y,U):=\{\gamma\in S^* \ : \ \gamma\cdot y\in U\}$ is internally syndetic, that is, $S^{**}\subseteq (S^*)^{-1}R(y,U)^*$.  Thus, there is $\gamma\in S^*$ such that $\gamma\cdot \beta^*\in R(y,U)^*$.  Let $\alpha\in L$ be such that $\alpha\sim \gamma\cdot \beta^*$, so $\alpha\in R(y,U)$.  We then have that $\alpha\cdot y\in U$, whence $\tT_\alpha(y)=y$.

(2) implies (3):  Take $\beta\in L$ such that $T_\beta(y)=y$.  Take idempotent $\alpha \in L$ such that $\alpha\cdot \beta^*\sim\beta$.  (See the discussion before the statement of the theorem.)  We then have
$$\tT_\alpha(y)=\tT_\alpha(\tT_\beta(y))=\tT_{\alpha\cdot \beta^*}(y)=\tT_\beta(y)=y.$$

(3) implies (4) is trivial.

(4) implies (1):  Suppose that $\tT_\alpha(x)=y$.  Note then that
$$\tT_\alpha(y)=\tT_{\alpha}(\tT_\alpha(x))=\tT_{\alpha\cdot \alpha^*}(x)=\tT_\alpha(x)=y.$$  Fix a neighborhood $U$ of $y$ and take a neighborhood $V$ of $y$ so that $\overline{V}\subseteq U$.  Set $A:=R(y,V)$. Since $\alpha\in A^*$, we have that $A_\alpha$ is syndetic, whence it suffices to show that $A_\alpha\subseteq R(y,U)$.  But if $s\in A_\alpha$, then $s\alpha\in A^*$, so $(s\alpha)\cdot y\in V^*$, so 
$$T_s(y)=T_s(\tT_{\alpha}(y)=\tT_{s\alpha}(y)\in \overline{V}\subseteq U,$$ as desired.
\end{proof}

\begin{cor}
For $x,y\in X$, the following are equivalent:
\begin{enumerate}
\item $x$ and $y$ are proximal and $y$ is uniformly recurrent.
\item There is a minimal idempotent $\alpha$ such that $\tT_\alpha(x)=y$.
\end{enumerate}
\end{cor}

\begin{proof}
(1) implies (2):  Let $L$ be a minimal left u-ideal contained in $I(x,y)$.  By above, there is idempotent $\alpha\in L$ such that $\tT_\alpha(y)=y$, so $\tT_\alpha(x)=\tT_\alpha(y)=y$.

(2) implies (1):  Obvious from above.
\end{proof}

We are now ready to prove the main result of this section:

\begin{thm}
$A\subseteq S$ is central if and only if it is dynamically central.
\end{thm}

\begin{proof}
First suppose that $A$ is central.  Let $Q:=S\cup \{e\}$ where $e$ is a new element that acts as a two-sided identity for $Q$.  Let $X=2^Q$ and $(T_s)_{s\in S}$ be as in Lemma \ref{Bernoulli}.  We show that $A$ is a dynamically central subset of $S$ as witnessed by this dynamical system.  Let $x\in Q$ be the characteristic function of $A$.  Let $\alpha\in S^*$ be a minimal idempotent such that $\alpha\in A^*$.  Set $y:=\tT_\alpha(x)$.  Then we know that $x$ and $y$ are proximal and $y$ is uniformly recurrent.  Set $U:=\{z\in X \ : \ z(e)=y(e)\}$, a neighborhood of $y$ in $X$.  It suffices to show that $A=R(x,U)$.  First note that $y=\st(\alpha\cdot x)$ implies that there is $s\in A$ such that $s\cdot x\in U$, whence it follows that $y(e)=(s\cdot x)(e)=x(es)=x(s)=1$.  It follows that, for $s\in S$, we have $s\in A \Leftrightarrow x(s)=1\Leftrightarrow T_s(x)(e)=1\Leftrightarrow T_s(x)\in U$, as desired.  

Now suppose that $A$ is dynamically central, so there are $x,y\in X$ that are proximal, $y$ is uniformly recurrent, and there is a neighborhood $U$ of $y$ such that $A=R(x,U)$.  Take a minimal idempotent $\alpha$ such that $\tT_\alpha(x)=y$.  Since $\alpha\cdot x\approx y$, we have that $\alpha\cdot x\in U^*$, and hence $\alpha\in A^*$.  It follows that $A$ is central.
\end{proof}

\begin{example}[Exercise 19.3.2 in \cite{HS}]
Suppose that $A,B\subseteq \omega$.  By considering the dynamical system $2^\omega$ as above, we see that:
\begin{enumerate}
\item $A$ is uniformly recurrent if, for every $k\in \bb N$, the set 
$$\{n\in \bb N \ : \ A+[n,n+k)=n+(A\cap [0,k))\}$$ is syndetic.
\item $A$ and $B$ are proximal if there are arbitrarily long intervals $I$ such that $A\cap I=B\cap I$.
\item $A$ is central if it is proximal to a uniformly recurrent set containing $0$. 
\end{enumerate}

Of course, we are applying the adjectives to a set when it applies to its characteristic function.  It is worth noting the nonstandard translation of the above:

\begin{enumerate}
\item $A$ is uniformly recurrent if and only if:  for every infinite $I$, there is $x\in I$ such that $(x+\bb N)\cap A^*=A$.
\item $A$ and $B$ are proximal if and only if there is an infinite interval $I$ such that $A^*\cap I=B^*\cap I$.
\end{enumerate}
\end{example}
%

\end{document}